\documentclass{amsart}
\usepackage{amsmath,amssymb,amsfonts}
\usepackage[margin=1in]{geometry}
\usepackage{algorithmic}
\usepackage{algorithm}
\usepackage{graphicx}
\usepackage{caption}
\usepackage{subcaption}
\usepackage{multirow}
\usepackage{textcomp}
\usepackage{float}
\usepackage{graphicx,color}
\usepackage{epstopdf}
\usepackage[section]{placeins}
\usepackage{enumerate}
\usepackage{lipsum}

\usepackage{amsthm}
\usepackage{hyperref}
\newtheorem{theorem}{Theorem}[section]
\newtheorem{corollary}{Corollary}

\newtheorem{lemma}[theorem]{Lemma}
\newtheorem{proposition}{Proposition}

\theoremstyle{definition}
\newtheorem{definition}[theorem]{Definition}

\author{Benjamin P. Russo$^{1*}$}
\thanks{$^{1}$ Farmingdale State College (SUNY), Department of Mathematics, \href{}{russobp@farmingdale.edu}}

\author{Rushikesh Kamalapurkar$^{2\dagger\ddagger}$}
\thanks{$^2$ Oklahoma State University, School of Mechanical and Aerospace Engineering, \href{}{rushikesh.kamalapurkar@okstate.edu}}

\author{Dongsik Chang$^{3\star}$}
\thanks{$^3$ Oregon State University, School of Mechanical, Industrial, and Manufacturing Engineering, \href{}{changdo@oregonstate.edu}}

\author{Joel A. Rosenfeld$^{4\dagger\ddagger}$}
\thanks{$^4$ University of South Florida, Department of Mathematics and Statistics, \href{}{rosenfeldj@usf.edu}}
\thanks{$^*$ Corresponding author}
\thanks{$^{\star}$ This work was partially completed while the author was at Michigan State University}
\thanks{$^\dagger$ Research supported by AFOSR Award FA9550-20-1-0127.}
\thanks{$^\ddagger$ Research supported by NSF grant ECCS-2027976.}

\title{Motion Tomography via Occupation Kernels}
\date{}
\begin{document}
\maketitle

\begin{abstract}
The goal of motion tomography is to recover a description of a vector flow field using information on the trajectory of a sensing unit. 
 In this paper, we develop a predictor corrector algorithm designed to recover vector flow fields from trajectory data with the use of occupation kernels developed by Rosenfeld et al.  \cite{rosenfeld2019dynamic,rosenfeld2019occupation}. Specifically, we use the occupation kernels as an adaptive basis; that is, the trajectories defining our occupation kernels are iteratively updated to improve the estimation on the next stage. Initial estimates are established, then under mild assumptions, such as relatively straight trajectories, convergence is proven using the Contraction Mapping Theorem. We then compare to the established method by Chang et al. \cite{SCC.Chang.Wu.ea2017} by defining a set of error metrics. We found that for simulated data, which provides a ground truth, our method offers a marked improvement and that for a real-world example we have similar results to the established method. 
\end{abstract}

\section{Introduction}
Over the past decade, unmanned aircraft and underwater systems have evolved significantly and are on the verge of becoming a ubiquitous part of urban and littoral landscape. To compensate for the lack of access to the global positioning system (GPS), unmanned underwater vehicles (UUVs) often rely on the knowledge of the flow field to improve localization \cite{SCC.Petrich.Woolsey.ea2009}. Similarly, accurate estimation of urban wind fields is widely acknowledged to be a significant challenge for implementation of traffic control systems (such as \cite{SCC.Aweiss.Owens.ea2018}) for unmanned air vehicles (UAVs) \cite{SCC.Stepanyan.Krishnakumar2017}. While it is possible to model air flow fields and ocean currents using measurements from on-board sensors, lack of accurate localization (in the case if UUVs) and vechicle-induced noise (in the case of UAVs, especially multi-rotor UAVs), creates significant challenges in acquisition and processing of the data generated by on-board sensors. The aforementioned challenges, along with the payload reduction associated with removing flow sensors, motivates the development of estimation techniques that rely only on the effect of the flow field on the motion of UAVs and UUVs, and not on direct measurements of the flow velocities. 

Motion tomography refers to the reconstruction of a vector field using its accumulated effects on mobile sensing units as they travel through the field \cite{wu2013glider,SCC.Chang.Wu.ea2017}. Motion tomography allows for the use of low cost mobile underwater/air vehicles as sensors to accumulate sufficient data for estimation of vector fields resulting from wind and ocean currents. As a result, military applications such as ocean current mapping for effective navigation of mine countermeasure UUVs in littoral environments, commercial applications such as wind field mapping for navigation of small package delivery UAVs, and disaster response applications such as wind field mapping for the prediction of flame front propagation and smoke spread, stand to benefit from fast and accurate motion tomography. In this paper we propose an algorithm for motion tomography based on occupation kernels developed in \cite{rosenfeld2019dynamic, rosenfeld2019occupation}. We show this algorithm converges via the contraction mapping theorem. 

The developed approach to motion tomography has several advantages over existing techniques such as \cite{SCC.Chang.Wu.ea2017}. The flow field is approximated here using the occupation kernels as basis functions for approximation, whereas \cite{SCC.Chang.Wu.ea2017} requires a piecewise constant description of the flow field or a parameterization with respect to Gaussian RBFs. Moreover, \cite{SCC.Chang.Wu.ea2017} employs a renormalization routine which imposes limitations on the motion of the mobile sensors. The proposed occupation kernel method avoids the renormalization and does not add further restrictions on the motion. Finally, the representation of the flow field with respect to the occupation kernel basis allows for the application of the approximation abilities of RKHSs, which are exploited in the convergence analysis.

\section{Tools}
A reproducing kernel Hilbert space (RKHS), $H$, over a set $X$ is a Hilbert space of real valued functions over the set $X$ such that for all $x \in X$ the evaluation functional $E_x g := g(x)$ is bounded. As such, the Riesz representation theorem guarantees, for all $x \in X$, the existence of a function $k_x \in H$ such that $\langle g, k_x \rangle_H = g(x)$, where $\langle \cdot, \cdot \rangle_H$ is the inner product for $H$ \cite[Chapter 1]{paulsen2016introduction}. The function $k_x$ is called the reproducing kernel function at $x$, and the function $K(x,y) = \langle k_y, k_x \rangle_H$ is called the kernel function corresponding to $H$.

\begin{definition}\label{def:occ}Let $X \subset \mathbb{R}^n$ be compact, $H$ be a RKHS of continuous functions over $X$, and $\gamma:[0,T] \to X$ be a bounded measurable trajectory. The functional $g \mapsto \int_0^T g(\gamma(\tau)) d\tau$ is bounded, and may be respresented as $\int_0^T g(\gamma(\tau)) d\tau = \langle g, \Gamma_{\gamma}\rangle_H,$ for some $\Gamma_{\gamma} \in H$ by the Riesz representation theorem. The function $\Gamma_{\gamma}$ is called the occupation kernel corresponding to $\gamma$ in $H$.\end{definition}

The value of an inner product against an occupation kernel in a RKHS can be approximated by leveraging quadrature techniques for integration. The numerical experiments described below utilize Simpson's rule when computing inner products involving occupation kernels and their associated Gram matrices, pivotal in the analysis contained in Theorem \ref{contraction mapping}. We present Theorem \ref{thm:simpsons-convergence} which shows convergence with Simpson's Rule.  Moreover, the occupation kernels themselves can be expressed as an integral against the kernel function in a RKHS as demonstrated in Proposition \ref{prop:integral-rep}. Theorem \ref{thm:simpsons-convergence} and Proposition \ref{prop:integral-rep} originally appear in \cite{rosenfeld2019occupation}, but are reproduced below for completeness.

\begin{proposition}\label{prop:integral-rep}
Let $H$ be a RKHS over a compact set $X$ consisting of continuous functions and let $\gamma : [0,T] \to X$ be a continuous trajectory as in Definition \ref{def:occ}. The occupation kernel corresponding to $\gamma$ in $H$, $\Gamma_{\gamma}$, may be expressed as 
\begin{equation}\label{eq:integral-rep}\Gamma_\gamma(x) = \int_0^T K(x,\gamma(t)) dt.\end{equation}
\end{proposition}

\begin{proof}
Note that $\Gamma_\gamma(x) = \langle \Gamma_\gamma, K(\cdot,x)\rangle_H$, by the reproducing property of $K$. Consequently,
\begin{gather*}
    \Gamma_\gamma(x) = \langle \Gamma_\gamma, K(\cdot,x)\rangle_H= \langle K(\cdot,x), \Gamma_\gamma \rangle_H
     = \int_0^T K(\gamma(t),x) dt
    = \int_0^T K(x,\gamma(t)) dt,
\end{gather*}
which establishes the result.
\end{proof}

Leveraging Proposition \ref{prop:integral-rep}, quadrature techniques can be demonstrated to give not only pointwise convergence but also norm convergence in the RKHS, which is a strictly stronger result.

\begin{theorem}\label{thm:simpsons-convergence}
Under the hypothesis of Proposition \ref{prop:integral-rep}, let $t_0 = 0 < t_1 < t_2 < \ldots < t_{F} = T$ (with $F$ even and $t_i$ evenly spaced), suppose that $\gamma$ is a fourth order continuously differentiable trajectory and $H$ is composed of fourth order continuously differentiable functions. Set $h$ to satisfy $t_i = t_0 + ih$, and consider 
\begin{gather}\label{eq:quadrature-rep-simpsons}
\hat \Gamma_\gamma(x) := \frac{h}{3}\left(K(x,\gamma(t_0)) + 4 \sum_{i=1}^{\frac{F}{2}} K(x,\gamma(t_{2\cdot i - 1}))
+ 2 \sum_{i=1}^{\frac{F}{2}-1} K(x,\gamma(t_{2\cdot i})) + K(x,\gamma(t_F)) \right).
\end{gather}
The norm distance is bounded as $\| \Gamma_{\gamma} - \hat \Gamma_{\gamma} \|_H^2 = O({\color{black}h^4}).$
\end{theorem}

\begin{proof}
Consider $\| \Gamma_{\gamma} - \hat \Gamma_{\gamma} \|_H^2 = \langle \Gamma_\gamma, \Gamma_\gamma \rangle_H + \langle \hat\Gamma_\gamma, \hat\Gamma_\gamma \rangle_H - 2 \langle \Gamma_\gamma, \hat\Gamma_\gamma \rangle_H$. The term $\langle \hat\Gamma_\gamma, \hat\Gamma_\gamma \rangle_H$ is an implementation of the two-dimensional Simpson's rule (cf. \cite{burden2001numerical}) while $\langle \Gamma_\gamma, \Gamma_\gamma \rangle_H$ is the double integral $\int_0^T \int_0^T K(\gamma(t),\gamma(\tau)) dt d\tau.$ Thus, \[\langle \hat\Gamma_\gamma, \hat\Gamma_\gamma \rangle_H = \langle \Gamma_\gamma, \Gamma_\gamma \rangle_H + O({\color{black}h^4}).\]
Similarly, $\langle \Gamma_{\gamma},\hat \Gamma_{\gamma}\rangle_H$ integrates in one variable while implementing Simpson's rule in the other. Consequently, \[\langle \Gamma_{\gamma},\hat \Gamma_{\gamma}\rangle_H = \langle \Gamma_{\gamma}, \Gamma_{\gamma}\rangle_H + O({\color{black}h^4}).\]
The conclusion of the theorem follows. %{\color{black}Look up convergence rates for Simpson's rule to confirm.}
\end{proof}
It should be noted that the above convergence rate is for norm convergence and point-wise convergence can be demonstrated to be faster.  
 
\section{Problem Setup and Developed Algorithm}\label{setup}
Let $r:[0,T]\rightarrow \mathbb{R}^2$ represent a continuous trajectory for a mobile sensor attempting to travel in a straight line, but subject to an unknown flow field, $F:R\subset \mathbb{R}^2\rightarrow \mathbb{R}^2$, where $R$ is a compact subset of the plane. Let $\dot r = s \begin{pmatrix} \cos(\theta) & \sin(\theta) \end{pmatrix}^T + F(r)$, for a positive constant $s$, represent the true dynamics induced by the flow field. We will assume $F:R\rightarrow \mathbb{R}^2$ is locally Lipschitz in order to assure uniqueness of the solutions \cite{Wolfgang}. As the flow field is unknown, the anticipated dynamics are given as $\dot{\tilde{r}} = s \begin{pmatrix} \cos(\theta) & \sin(\theta) \end{pmatrix}^T$. After a mobile sensor has traveled through the flow field over a time period $[0,T]$, during which there is limited to no knowledge of the mobile sensor's position, the difference between the actual location of the mobile sensor, $r(T)$, and the anticipated location, $\tilde r(T)$, is given as 
\[D = r(T) - \tilde r(T) = \int_0^T \left( \dot r(t) - \dot{\tilde{r}}(t)\right) dt = \int_0^T F(r(t)) dt 
%= \langle F, \Gamma_r \rangle_H.
\]
Hence, $D=\langle F, \Gamma_r \rangle_H$ and the difference between the final locations of the mobile sensor describes the integral of the flow field along the trajectory $r$. This integral provides a type of tomographic sample of $F$. As the trajectory $r$ is treated as unknown, an approximation of $F$ using the sample generated by $\Gamma_{r}$ is difficult to assess directly. This motivates the developed iterative algorithm to determine the flow field, $F$, as well as the true trajectories, $r$.

Let $\{ s_i \}_{i=1}^M, \{ \theta_i \}_{i=1}^M$ be a collection of speeds and angles used to generate a collection of anticipated trajectories, i.e. trajectories governed by the dynamics $\dot{r}=s_i\begin{pmatrix}\cos(\theta_i)& \sin(\theta_i)\end{pmatrix}^\top$. Moreover, let $\{p_i\}_{i=1}^M\subset \mathbb{R}^2$ represent the starting point of the trajectories.
%%%%%%%%%%%%%%%%%%%%%%%%%%%%%%%%%%%%%%%%%%%%%%%%%%%%%%%%%%%%%%%%%%%%%%%%%%%%%%%%%%%%%%%%%%%%%%%%%%%%%%%%%%%%%%%%%%%%%%%%%%%%%%%%%%%%%%%%%%%%%%%%%%%%%%%%%%%%%%%%%%%%%%%
\begin{algorithm}
\caption{Iterative Motion Tomography Algorithm}
\begin{algorithmic}
\label{alg}
\STATE{Define $N$ as the number of iterates}
\STATE{Input: Samples $r_i(T)$ \quad $i\in\{1,\ldots,M\}$}

% \FOR{for $i\in\{1,\ldots,M\}$}
% \STATE{Generate via a numerical method $\tilde{r}_{i,0}:[0,1]\rightarrow R$, the unique solution to \[\dot{p}=s_i\begin{pmatrix}\cos(\theta_i)& \sin(\theta_i)\end{pmatrix}^\top,\quad p(0)=p_i\]
%  }
% \ENDFOR
\STATE{Set $D_{i,0}=r_i(T)-\tilde{r}_{i,0}(T)$\quad $i\in\{1,\ldots,M\}$}
\STATE{ Set $\hat{F}_{0}=0$}\\
\FOR{$n \in \{0,\ldots,N\}$} 
% \STATE{Input $\hat{F}_n=\sum_{i=1}^Mw_{i,n}\Gamma_{\tilde{r}_{i,n}}$}
\FOR{ $i\in\{1,\ldots,M\}$}
 \STATE{Generate via a numerical method $\tilde{r}_{i,n}:[0,1]\rightarrow R$ the unique solution to \[\dot{p}=s_i\begin{pmatrix}\cos(\theta_i)& \sin(\theta_i)\end{pmatrix}^\top+\hat{F}_n(p),\quad p(0)=p_i.\] }
\ENDFOR
\STATE{Set $D_{i,n}=r_i(T)-\tilde{r}_{i,n}(T)$\quad $i\in\{1,\ldots,M\}$}
\FOR{$i\in\{1,\ldots,M\}$}
\STATE{Compute $\hat{F}_{n+1}=\sum_{i=1}^M w_{i,n+1}\Gamma_{\tilde{r}_{i,n}}$ by solving
\[ 
\begin{pmatrix} \langle \Gamma_{\tilde r_{i,n}}, \Gamma_{\tilde r_{j,n}} \rangle \end{pmatrix}^{M,M}_{i,j=1}\begin{pmatrix}
w_{1,n+1} \\ \vdots \\ w_{M,n+1}
\end{pmatrix}=\begin{pmatrix}
D_{1,n} + \langle \hat F_n, \Gamma_{\tilde r_{1,n}} \rangle\\ \vdots \\ D_{M,n} + \langle \hat F_n, \Gamma_{\tilde r_{M,n}}\rangle
\end{pmatrix}.
\]}
\ENDFOR
\STATE{
Output $\hat{F}_{n+1}=\sum_{i=1}^M w_{i,n+1}\Gamma_{\tilde{r}_{i,n}}$}

\ENDFOR
\end{algorithmic}
\end{algorithm}

%%%%%%%%%%%%%%%%%%%%%%%%%%%%%%%%%%%%%%%%%%%%%%%%%%%%%%%%%%%%%%%%%%%%%%%%%%%%%%%%%%%%%%%%%%%%%%%%%%%%%%%%%%%%%%%%%%%%%%%%%%%%%%%%%%%%%%%%%%%%%%%%%%%%%%%%%%%%%%%%%%%%%%%

Significantly, after the initial data collection period, no further experiments are necessary to approximate the flow field. Specifically, Algorithm \ref{alg} only needs to produce new simulations of the approximate trajectories which ultimately converge to the true trajectories from the initial experiment.
\section{Convergence of Algorithm}
This section is devoted to establishing the necessary conditions for the above algorithm to converge. The main theoretical crux will be the contraction mapping theorem. Let $R\subset\mathbb{R}^2$ be compact, $F:R\rightarrow \mathbb{R}^2$ a Lipschitz continuous vector field, let $\{r_i\}_{i=1}^\infty$ be a countable set of trajectories, $r_i:[0,T]\rightarrow R$, satisfying 
 \[\dot{p}=s_i\begin{pmatrix}\cos(\theta_i)&\sin(\theta_i)\end{pmatrix}^\top +F(p)\quad p(0)=p_i\] 
 for some countable collection of $(s_i,\theta_i)\subset \mathbb{R}\times [0,2\pi)$, and countable set $\{p_i\}_{i=1}^\infty\subset R$. Let $H(R)$ be a reproducing kernel Hilbert space of $\mathbb{R}^2$ valued functions. 
\begin{theorem}\label{contraction mapping}
Let $\Gamma\subset H(R)$ be the subset of $\phi(x)=\begin{pmatrix}\phi_1(x)&\phi_2(x)\end{pmatrix}^\top$ with $\phi_i(x)=\sum_{j=1}^sw^i_{j,\phi}\Gamma_{r_{j,\phi}}(x)$ for $i=1,2$. We define $A:\Gamma\subset H(R)\rightarrow H(R)$ as follows: Given, $\phi(x)\in \Gamma$ define $\hat{r}_{j,\phi}$ to be the solution to
\[\dot{p}=s_j\begin{pmatrix}\cos(\theta_j)& \sin(\theta_j)\end{pmatrix}^\top+\phi(p)\quad p(0)=p_j.\]
Moreover, define $D^i_{j,\phi}:=r^i_j(T)-\hat{r}^i_{j,\phi}(T)$ for $i=1,2$. Then 
\[A(\phi)_i:=\sum_{j=1}^s\hat{w}^i_{j,\phi}\Gamma_{\hat{r}_{j,\phi}}(x)\]
where $\hat{w}^i_{j,\phi}$ are defined so that $A(\phi)_i$ satisfies
\[\langle A(\phi)_i,\Gamma_{\hat{r}_{j,\phi}}\rangle_{H}=D^i_{j,\phi}+\langle \phi, \Gamma_{\hat{r}_{j,\phi}}\rangle_{H}\]
for $i=1,2$.
There exists a closed finite diameter subspace $E\subset \Gamma$ such that $A|_{E}$ is a contraction mapping and thus extends to a contraction mapping on $H(R)$. 
\end{theorem}
We note that Hilbert spaces have the Lipschitz extension property \cite{geo-nonlinear-analysis}, i.e. a Banach space $E$ has the Lipschitz (Contraction) extension property if every Lipschitz map on a subset of $E$ extended to all of $E$ with the same Lipschitz constant. Moreover, we will also not be proving existence of the solution, but rather that the mapping described above is a contraction in some ball of diameter less than one containing the solution and applying a version of Contraction Mapping Theorem \cite{comp-functional-analysis}to prove convergence. Furthermore, we will also be suppressing the index $i=1,2$, since this would make the proof unnecessarily messy notationally and since the proof would work over any finite size dimension. 
\begin{theorem}[Contraction Mapping Theorem]
Let $X$ be a complete metric space and let $S_0$ be a closed subset of $X$ of finite diameter. Let $P_0:S_0\rightarrow S_0$ be a contraction mapping. Then the sequence of iterates $\{x_k\}$ produced by successive iterations $x_{k+1}=P(x_k)$ converges to $x=P(x)$, the unique fixed point of $P$ in $S_0$ for any $x_0$ in $S_0$. 
\end{theorem}
 To prove Theorem \ref{contraction mapping} we establish several propositions that give inequalities relating norms to the norms of the flow fields. 
\begin{proposition}\label{diffeqnormbound}
Let $H(R)$ be the reproducing kernel Hilbert space described above. For $\phi,\psi\in H(R)$, if $r_\phi$ and $r_\psi$ are the unique solutions to the initial value problems
\[\dot{p}=s\begin{pmatrix}\cos(\theta)&\sin(\theta)\end{pmatrix}^\top +\phi(p),\quad p(0)=p\]
\[\dot{p}=s\begin{pmatrix}\cos(\theta)&\sin(\theta)\end{pmatrix}^\top +\psi(p),\quad p(0)=p\in R\]
respectively, then 
\[|r_{\phi}(t)-r_{\psi}(t)|_2\leq M\|\phi-\psi\|_H\]
for some constant $M$. 
\end{proposition}
\begin{proof}
We have that 
\begin{align*}
|\dot{r}_\phi-\dot{r}_\psi|=|\phi(r_\phi(t))-\psi(r_{\psi}(t))|_2&\leq |\phi(r_\phi(t))-\phi(r_{\psi}(t))|_2+|\phi(r_\psi(t))-\psi(r_{\psi}(t))|_2\\
    &\leq \|\nabla\phi\|_2\cdot|r_\phi(t)-r_\psi(t)|+\|\phi-\psi\|_H\sqrt{K(r_\psi(t),r_\psi(t))} 
\end{align*}
by an application of mean value theorem. It follows that,
\[|r_\phi(t)-r_\psi(t)|\leq \exp(\nabla\phi)|r_\phi(0)-r_\psi(0)|+\|\phi-\psi\|_H\cdot M=\|\phi-\psi\|_H\cdot M
\]
for some $M$. 
\end{proof}
The proof of the following proposition can be found in \cite{rosenfeld2019occupation}, but is presented here for clarity. 
\begin{proposition}\label{pathdiff}
Suppose $H$ is a RKHS over a set $X$ consisting of continuous functions and let $\gamma_1(t)$ and $\gamma_2(t)$ be two trajectories with homotopy $\{\gamma_s(t)\}$. The map 
$s\mapsto \Gamma_{\gamma_s}$
is continuous. 
\end{proposition}
\begin{proof}
As $[0,T] \times [0,1]$ is compact, the map $(t,s) \mapsto \gamma_s(t)$ is uniformly continuous. That is for every $\varepsilon > 0$ there exists a $\delta > 0$ such that whenever $\|(t_1,s_1)-(t_s,s_2)\|_2 < \delta$, $\| \gamma_{s_1}(t_1) - \gamma_{s_2}(t_2) \|_2 < \epsilon.$ Moreover, as $K(\cdot,\cdot)$ is continuous, and the image of $\gamma_{s}(t)$ is compact, the map $(s_1,t,s_2,\tau) \mapsto K(\gamma_{s_1}(t),\gamma_{s_2}(\tau))$ is uniformly continuous.

Fix $\varepsilon > 0$ and select $\delta > 0$ such that 
\begin{align*}
    |K(\gamma_{s_1}(t),\gamma_{s_1}(\tau)) - K(\gamma_{s_2}(t),\gamma_{s_1}(\tau))|& < \frac{\varepsilon}{2T^2} \quad \text{ and}\\
    |K(\gamma_{s_2}(t),\gamma_{s_2}(\tau)) - K(\gamma_{s_2}(t),\gamma_{s_1}(\tau))|& < \frac{\varepsilon}{2T^2}
\end{align*}
whenever $|s_1 - s_2| < \delta$. Select $s_1, s_2$ such that $|s_1 - s_2| < \delta$, then
\begin{align}\label{homocontinuity}
\|\Gamma_{s_1} - \Gamma_{s_2} \|_H^2 &=  \langle \Gamma_{s_1}, \Gamma_{s_1} \rangle + \langle \Gamma_{s_2}, \Gamma_{s_2} \rangle - 2 \langle \Gamma_{s_1}, \Gamma_{s_2} \rangle  \\
&= \int_0^T \int_0^T ( K(\gamma_{s_1}(t),\gamma_{s_1}(\tau)) - K(\gamma_{s_2}(t),\gamma_{s_1}(\tau)) dt d\tau\nonumber \\
&+\int_0^T \int_0^T K(\gamma_{s_2}(t),\gamma_{s_2}(\tau)) - K(\gamma_{s_2}(t),\gamma_{s_1}(\tau)) ) dt d\tau,\nonumber 
\end{align}
Note that Equation \eqref{homocontinuity} is positive and bounded by $\varepsilon$ by construction. Hence, the map $s \mapsto \Gamma_{s}$ is continuous.
\end{proof}
\begin{corollary}\label{normdiffGammas}Following the notational convention defined in Proposition \ref{pathdiff}, if $|r_{\phi}(t)-r_{\psi}(t)|\leq M\|\phi-\psi\|_H$ for some $M$ then $\|\Gamma_{r_{\phi}}-\Gamma_{r_{\psi}}\|\leq KT^2\|\phi-\psi\|^{1/2}_H$ as well for some constant $C$ independent of $T$.
\end{corollary}
\begin{proof}
\begin{align}\label{homo-equality}
\| \Gamma_{r_\phi} - \Gamma_{r_\psi} \|_H^2 &=  \langle \Gamma_{r_\phi}, \Gamma_{r_\phi} \rangle + \langle \Gamma_{r_\psi}, \Gamma_{r_\psi} \rangle - 2 \langle \Gamma_{r_\phi}, \Gamma_{r_\psi} \rangle  \\
&= \int_0^T \int_0^T  K({r_\phi}(t),{r_\phi}(\tau)) - K({r_\psi}(t),{r_\phi}(\tau))\,dt\, d\tau\nonumber \\
&+\int_0^T \int_0^T  K({r_\psi}(t),{r_\psi}(\tau)) - K({r_\psi}(t),{r_\phi}(\tau))  \,dt\, d\tau\nonumber \\
&\leq \int_0^T \int_0^T \left(|\nabla_y K(r_\phi(t),\xi_\psi)|\right)\cdot|r_\phi(\tau)-r_\psi(\tau)|\, dt\, d\tau\nonumber\\
&+\int_0^T \int_0^T \left(|\nabla_y K(r_\psi(t),\xi_\phi)|\right)\cdot|r_\phi(\tau)-r_\psi(\tau)|\, dt\, d\tau\nonumber\\
&\leq CT^2\|\phi-\psi\|_H\nonumber.
\end{align}
The last two inequalities are due to an application of mean value theorem and proposition \ref{diffeqnormbound}.
\end{proof}

In order to talk about the stability of the interpolation process we define the following:
\begin{definition}
The trajectory separation distance is given by
\[q_{X,T}:=\frac{1}{2}\min_{\stackrel{j\neq k}{t,\tau\in[0,T] }}\|\gamma_j(t)-\gamma_k(\tau)\|_2\]
\end{definition}
\noindent This is the maximum radius such that all ``tubes" centered on the trajectories are disjoint. 

The next theorem will be a trajectory variant of a theorem found in Wendland \cite{Wendland}. In Wendland, a bound for the minimum eigenvalue is obtained by expressing the kernel function $\Phi(x-y)$ in terms of its Fourier transform and comparing with an intermediate radial functions $\Psi_M$ based on characteristic functions for balls of radius $M$. This comparison is given in terms of the separation distance between centers. By defining the trajectory separation distance as above the inequalities found in Wendland remain valid since the points in the image of each trajectory are separated by at least $q_{X,T}$. 

In the following theorem let $G_\Gamma$ be the Occupation kernel Grammian, i.e. for a set of trajectories, $r_i:[0,T]\rightarrow \mathbb{R}^2$, with $1\leq i\leq N$
\[
G_\Gamma:=\left(\langle \Gamma_{r_j}, \Gamma_{r_k}\rangle\right)_{j,k=1}^N
=\int_0^T\int_0^T\left(K(r_j(t),r_k(\tau))\right)_{j,k=1}^N\, dt\,d\tau.
\]
\begin{theorem}\label{spectralbound}
With $\Phi(x-y)=K(x,y)=\exp(-\mu\|x-y\|_2^2)$ the minimal eigenvalue of the Occupation kernel Grammian is bounded by
\[\lambda_{\min(G_\Gamma)}\geq \frac{C_2}{2\mu}\frac{\exp(-M_2^2/(q_{X,T}^2\mu))T^2}{q_{X,T}^2}\]
with $M_2=12\sqrt[3]{\frac{\pi}{9}}$ and $C_2=\frac{M^2}{16}$.
\end{theorem}
\begin{proof}
Eventually, we will have $d=2$ but for now we will write generally. 
By linearity, for $\alpha=(\alpha_1, \ldots, \alpha_N)^\top$, 
given, 
\begin{equation}\label{eq:fourier}
    \alpha^\top G_{\varphi(r)}\alpha =\sum_{j,k=1}^N\alpha_j\alpha_k\varphi(r_j(t)-r_k(\tau))
    =\int_{\mathbb{R}^d}\sum_{j,k=1}^N\alpha_j\alpha_ke^{i\omega^\top(r_j(t)-r_k(t))}\hat{\varphi}(\omega)\, d\omega\nonumber
\end{equation}
by linearity an application of Fubini we have,
\begin{align}\label{eq: G_Gammainnerprod}
\alpha^\top G_\Gamma\alpha &=\int_0^T\int_0^T \alpha^\top G_{\varphi(r)}\alpha \, dt\, d\tau\\
&=\int_0^T\int_0^T \int_{\mathbb{R}^d} \sum_{j,k=1}^N\alpha_j\alpha_ke^{i\omega^\top(r_j(t)-r_k(\tau))}\hat{\varphi}(\omega)\, d\omega\nonumber\\
&=\int_{\mathbb{R}^d}\left|\sum_{j=1}^N\alpha_j \int_0^Te^{i\omega^\top r_j(t)}\,dt\right|^2\hat{\varphi}(\omega)\,d\omega\nonumber
\end{align}
where \eqref{eq:fourier} is valid whenever $\varphi:\mathbb{R}^d\rightarrow \mathbb{C}$ is a continuous and slowly increasing function possessing a non-negative generalized Fourier transform. 
Define, 
\[\psi_M(x)=\frac{\varphi_0(M)\Gamma\left(\frac{d}{2}+1\right)}{2^{d/2}}\|x\|_2^{-d}J_{d/2}^2(M\|x\|_2)
\]
where $J_\nu$ is the Bessel function of the first kind and \[\varphi_0(M)=\inf_{\|\omega\|_2\leq 2M}\hat{\varphi}(\omega).\]
In Wendland \cite{Wendland} it is established that for a collection of points $x_i$, $i=1,\ldots N$,  
\begin{align*}
\int_{\mathbb{R}^d}\sum_{j,k=1}^N\alpha_j\alpha_ke^{i\omega^\top(x_j-x_k)}\hat{\psi}_M(\omega)\, d\omega &=\sum_{j,k=1}^N\alpha_j\alpha_k\psi_M(x_j-x_k)\\
&\geq \|\alpha\|_2^2\left(\psi_M(0)-\max_{1\leq j\leq N}\sum_{k=1, k\neq j}^N |\psi_M(x_j-x_k)|\right)
\end{align*}
and that for the given value of $M$, 
\[\max_{1\leq j\leq N}\sum_{k=1, k\neq j}^N |\psi_M(x_j-x_k)|\leq \frac{1}{2}\psi_M(0).\]
Thus, for the stated value of $M$ we have that,
\begin{equation}\label{eq: Wendlandineq}
\sum_{j,k=1}^N\alpha_j\alpha_k\psi_M(x_j-x_k)\geq \|\alpha\|_2^2\frac{\psi_M(0)}{2}.
\end{equation}
Now, the key insight is that the establishment of \ref{eq: Wendlandineq} is unaffected by letting $x_j=r_j(t), x_k=r_k(\tau)$ for $t,\tau\in[0,T]$ by replacing the standard separation distance $q_X$ with $q_{X,T}$ and that this holds over $[0,T]$. 

Our auxiliary function $\psi_M$ is designed so that $\hat{\psi}_M(\omega)\leq \hat{\varphi}(\omega)$, giving us
\begin{equation}\label{eq:phi-psi ineq}
\int_{\mathbb{R}^d}\left|\sum_{j=1}^N\alpha_j \int_0^Te^{i\omega^\top r_j(t)}\,dt\right|^2\hat{\varphi}(\omega)\,d\omega
\geq \int_{\mathbb{R}^d}\left|\sum_{j=1}^N\alpha_j \int_0^Te^{i\omega^\top r_j(t)}\,dt\right|^2\hat{\psi}_M(\omega)\,d\omega.
\end{equation}
By combining equations, \eqref{eq:fourier}, \eqref{eq: G_Gammainnerprod}, \eqref{eq: Wendlandineq}, and \eqref{eq:phi-psi ineq}
we get
\begin{align*}
\alpha^\top G_\Gamma\alpha&=\int_0^T\int_0^T \alpha^\top G_{\varphi(r)}\alpha \, dt\, d\tau=\int_0^T\int_0^T \int_{\mathbb{R}^d} \sum_{j,k=1}^N\alpha_j\alpha_ke^{i\omega^\top(r_j(t)-r_k(\tau))}\hat{\varphi}(\omega)\, d\omega\\
&\geq \int_0^T\int_0^T\|\alpha\|_2^2\frac{\psi_M(0)}{2}\, dt\, d\tau.
\end{align*}
Finally, 
\[\lambda_{\min(G_\Gamma)}=\inf\frac{\alpha^\top G_\Gamma\alpha}{\|\alpha\|_2^2}\geq \frac{C_2}{2\mu}\frac{\exp(-M_2^2/(q_{X,T}^2\mu))T^2}{q_{X,T}^2}\]
with $M_2=12\sqrt[3]{\frac{\pi}{9}}$ and $C_2=\frac{M_2^2}{16}$ and using the bounds for $\psi_M(0)$ established in \cite{Wendland}.
\end{proof}
 As a corollary to the above analysis we show that the condition number of our interpolation matrices is at least as good as the standard Gaussian kernels. 
\begin{corollary}\label{maxeigenvalue}
If \[G_\Gamma=\left(\langle \Gamma_{r_j}, \Gamma_{r_k}\rangle\right)_{j,k=1}^N=\int_0^T\int_0^T\left(K(r_j(t),r_k(\tau))\right)_{j,k=1}^N\, dt\,d\tau,\]
then 
\[\lambda_{\max}(G_\Gamma)\leq NT^2 \Phi(0)\]
where $\Phi(x_i-x_j)=K(x_j,x_i)$. 
\end{corollary}
\begin{proof}
By a standard application of Gershgorin's circle theorem on eigenvalues,
\[\lambda_{\max}(G_\Gamma)\leq N \max_{i,j=1,\ldots N}|\Phi(r_j(t)-r_k(\tau))|\leq N\int_0^T\int_0^T |\Phi(0)|\, dt\, d\tau=NT^2|\Phi(0)|
\]
where the above is due to the properties of positive semi-definite functions. 
\end{proof}
Since the condition number of an interpolation matrix $A$ is given by the ratio of the maximum and minimum eigenvalues, the presence of the $T^2$ in the above inequality nullifies the $1/T^2$ appearing in estimate for the minimal eigenvalue. In essence, this shows that can take the minimal distance between trajectories as a measurement of how good our approximations will be.
\subsection{Main Inequalities}
Following the notion established in \ref{contraction mapping} consider the following,
\begin{align}\label{maininequality}\|A(\phi)-A(\psi)\|_H
&=\left\|\sum_{j=1}^s\hat{w}_{j,\phi}\Gamma_{\hat{r}_{j,\phi}}(x)-\sum_{j=1}^s\hat{w}_{j,\psi}\Gamma_{\hat{r}_{j,\psi}}(x)\right\|_H\nonumber\\
&\leq \sum_{j=1}^s|\hat{w}_{j,\phi}|\|\Gamma_{\hat{r}_{j,\phi}}-\Gamma_{\hat{r}_{j,\psi}}\|_H+\sum_{j=1}^s|\hat{w}_{j,\phi}-\hat{w}_{j,\psi}|\|\Gamma_{\hat{r}_{j,\psi}}\|_H\
\end{align}
To proceed forward, we will need to establish control on terms above in as a function of $\|\phi-\psi\|_H$. 
Note by definition $\begin{pmatrix}\hat{w}_{1,\phi}  \cdots  \hat{w}_{s,\phi}\end{pmatrix}^\top$ satisfies
\begin{equation}\begin{pmatrix} \langle \Gamma_{\hat r_{1,\phi}}, \Gamma_{\hat r_{1,\phi}} \rangle &\cdots& \langle \Gamma_{\hat r_{s,\phi}}, \Gamma_{\hat r_{1,\phi}} \rangle\\
\vdots &\ddots& \vdots\\
\langle \Gamma_{\hat r_{1,\phi}}, \Gamma_{\hat r_{M,\phi}} \rangle&\cdots& \langle \Gamma_{\hat r_{M,\phi}}, \Gamma_{\hat r_{M,\phi}} \rangle\end{pmatrix}
\begin{pmatrix}
\hat{w}_{1,\phi} \\ \vdots \\ \hat{w}_{s,\phi}
\end{pmatrix}
=
\begin{pmatrix}
D_{1,\phi} + \langle \phi, \Gamma_{\hat r_{1,\phi}} \rangle\\ \vdots \\ D_{s,\phi} + \langle \phi, \Gamma_{\hat r_{s,\phi}}\rangle
\end{pmatrix}.
\end{equation}
For notational convenience, given a $\phi\in H(R)$ let 
\[G_\phi=\begin{pmatrix} \langle \Gamma_{\hat r_{1,\phi}}, \Gamma_{\hat r_{1,\phi}} \rangle & \cdots & \langle \Gamma_{\hat r_{s,\phi}}, \Gamma_{\hat r_{1,\phi}} \rangle\\
\vdots & \ddots & \vdots\\
\langle \Gamma_{\hat r_{1,\phi}}, \Gamma_{\hat r_{s,\phi}} \rangle & \cdots & \langle \Gamma_{\hat r_{s,\phi}}, \Gamma_{\hat r_{s,\phi}} \rangle\end{pmatrix}\]
Now,
\begin{align}\label{norm1}
\left\|\begin{pmatrix}\hat{w}_{1,\phi} \\ \vdots \\ \hat{w}_{s,\phi}\end{pmatrix}-\begin{pmatrix}\hat{w}_{1,\psi} \\ \vdots \\ \hat{w}_{s,\psi}\end{pmatrix}\right\|_2&\leq 
\|G_\phi^{-1}\|\left\|\begin{pmatrix}
D_{1,\phi}-D_{1,\psi} + \langle \phi, \Gamma_{\hat r_{1,\phi}} \rangle -\langle \psi, \Gamma_{\hat r_{1,\psi}}\rangle\\
\vdots \\ 
D_{s,\phi}-D_{s,\psi} + \langle \phi, \Gamma_{\hat r_{s,\phi}} \rangle -\langle \psi, \Gamma_{\hat r_{s,\psi}}\rangle\\
\end{pmatrix}\right\|_2\\
\label{norm2}
&+\|G_{\phi}^{-1}-G_{\psi}^{-1}\|
\left\|\begin{pmatrix}
D_{1,\psi} + \langle \psi, \Gamma_{\hat r_{1,\psi}} \rangle\\ \vdots \\ D_{s,\psi} + \langle \psi, \Gamma_{\hat r_{s,\psi}}\rangle\end{pmatrix}\right\|_2
\end{align}
A portion of the proof for Theorem\ref{contraction mapping} will be on establishing suitable control over of the norms in \eqref{norm1}, and \eqref{norm2}. Before proceeding, we will need the following lemma.
\begin{lemma}
Given a $\phi,\psi\in H(R)$, then \[\|G_\phi-G_\psi\|\leq ST^2\|\phi-\psi\|_H\]
for a constant $S\propto s$ and independent of $T$
\end{lemma}
\begin{proof}
Similar to the analysis in corollary \ref{maxeigenvalue}
\begin{align*}\lambda_{\max}(G_\phi-G_\psi)&\leq s \max_{i,j=1,\ldots s}|\langle \Gamma_{\hat r_{i,\phi}}-\Gamma_{\hat r_{i,\psi}}, \Gamma_{\hat r_{j,\phi}}-\Gamma_{\hat r_{j,\psi}} \rangle|\\
&\leq s\max_{i,j=1,\ldots,s}\|\Gamma_{\hat r_{i,\phi}}-\Gamma_{\hat r_{i,\psi}}\|_H\cdot\|\Gamma_{\hat r_{j,\phi}}-\Gamma_{\hat r_{j,\psi}}\|_H \\
&\leq ST^2\|\phi-\psi\|_H.
\end{align*}
Hence, by an application of corollary \ref{normdiffGammas} we have 
\[\lambda_{\max}(G_\phi-G_\psi)\leq sCT^2\|\phi-\psi\|_H.\]
\end{proof}
\subsection*{Analysis of the upper bounds:}
Given that by definition \[|D_{j,\phi}-D_{j,\psi}|=|\hat{r}_{j,\phi}(T)-\hat{r}_{j,\psi}(T)|\leq M\cdot\|\phi-\psi\|_H\] by Proposition \ref{diffeqnormbound}  and 
\[|\langle \phi,\Gamma_{\hat{r}_{j,\phi}}\rangle-\langle \psi,\Gamma_{\hat{r}_{j,\psi}}\rangle|\leq \|\phi\|_H\|\Gamma_{\hat{r}_{j,\phi}}-\Gamma_{\hat{r}_{j,\psi}}\|_H+\|\phi-\psi\|_H\cdot\|\Gamma_{\hat{r}_{j,\psi}}\|_H
\]
from inequality \ref{norm1}
%%%%%%%%%%%%%%%%%%%%%%%%%%%%%%%%%%%%%%%%%%%%%%%%%%%%%%%%%%%%%%%%%%%%%%%%%%%%%%%%%%%%%%%%%%%%%%%%%%%%%%%%%%%%%%%%%%%%%%%%%%%%%%%%%%%%%%%%%%%%%%%%%%%%%%%%%%%%%%%%
\begin{align}\label{estimate1}
&\|G_\phi^{-1}\|\left\|\begin{pmatrix}
D_{1,\phi}-D_{1,\psi} + \langle \phi, \Gamma_{\hat r_{1,\phi}} \rangle -\langle \psi, \Gamma_{\hat r_{1,\psi}}\rangle\\
\vdots \\ 
D_{s,\phi}-D_{s,\psi} + \langle \phi, \Gamma_{\hat r_{s,\phi}} \rangle -\langle \psi, \Gamma_{\hat r_{s,\psi}}\rangle\\
\end{pmatrix}\right\|_2\\
&\leq\frac{2\mu}{C_2}\frac{q_{X,T,\phi}^2}{\exp(-M_2^2/(q_{X,T,\phi}^2\mu))T^2}\cdot s\left[M\|\phi-\psi\|_H+\|\phi\|_H\sqrt{C}T\|\phi-\psi\|^{1/2}_H+\max\|\Gamma_{j,\psi}\|_H\|\phi-\psi\|_H\right]\nonumber
\end{align}
%%%%%%%%%%%%%%%%%%%%%%%%%%%%%%%%%%%%%%%%%%%%%%%%%%%%%%%%%%%%%%%%%%%%%%%%%%%%%%%%%%%%%%%%%%%%%%%%%%%%%%%%%%%%%%%%%%%%%%%%%%%%%%%%%%%%%%%%%%%%%%%%%%%%%%%%%%%%%%%%%%%%%%%%%%%%%%%%%%%%%%%%%%%%%%%%%%%%%%%%
By an application of Corollary \ref{normdiffGammas} and Theorem \ref{spectralbound}. Here we have used $q_{X,T,\phi}$ to denote the separation distance for the $\phi$-related trajectories. 

Turning our attention to inequality \eqref{norm2}
\begin{align}\label{estimate2}&\|G_{\phi}^{-1}-G_{\psi}^{-1}\|
\left\|\begin{pmatrix}
D_{1,\psi} + \langle \psi, \Gamma_{\hat r_{1,\psi}} \rangle\\ \vdots \\ D_{s,\psi} + \langle \psi, \Gamma_{\hat r_{s,\psi}}\rangle\end{pmatrix}\right\|_2\\
\leq& \|G_\phi\|^{-1}\|G_\phi-G_\psi\|\|G_\psi\|^{-1}\cdot\left\|\begin{pmatrix}
D_{1,\psi} + \langle \psi, \Gamma_{\hat r_{1,\psi}} \rangle\\ \vdots \\ D_{s,\psi} + \langle \psi, \Gamma_{\hat r_{s,\psi}}\rangle\end{pmatrix}\right\|_2\nonumber\\
\leq& \frac{4\mu^2}{C_2^2}\frac{q_{X,T,\phi}^2q_{X,T,\psi}^2}{\exp(-((M_2^2/q_{X,T,\phi}^2\mu)+(M_2^2/q_{X,T,\psi}^2\mu))T^4} \cdot sCT^2\|\phi-\psi\|_H\left\|\begin{pmatrix}
D_{1,\psi} + \langle \psi, \Gamma_{\hat r_{1,\psi}} \rangle\\ \vdots \\ D_{s,\psi} + \langle \psi, \Gamma_{\hat r_{s,\psi}}\rangle\end{pmatrix}\right\|_2\nonumber
\end{align}

\subsection*{Sufficient conditions for convergence:}
In order to prove convergence of our algorithm we will make some assumptions. Let $1>\varepsilon>0$, we will first assume that our iterative procedure starts reasonably close to the true solution $F$. In terms of the above, we are going to make the assumption that for both $\phi$ and $\psi$ we have $\|\phi-F\|_H\leq\varepsilon/2$, $\|\psi-F\|_H\leq\varepsilon/2$ and $\|\phi-\psi\|_H\leq\varepsilon<1$. This in turn gives us that $|D_{j,\phi}|,|D_{j,\psi}|<M\varepsilon$. We will also make an assumption on our underlying vector field. We are going to assume that the vector field has at-least the smoothness conditions stated in section \ref{setup} and is not overly powerful in comparison to the mobile sensor engines. In terms of the above, given the differential equation 
\[\dot{p}=s_j\begin{pmatrix}\cos(\theta_j)& \sin(\theta_j)\end{pmatrix}^\top+G(p)\quad p(0)=p_j\]
with solution $r_{G}(t)$, we will assume $\|\Gamma_{r_{G}}-\Gamma_{s}\|_H<\varepsilon$ where $s(t)$ is the straight line solution to the differential equation above with $G=0$, and that $\|\Gamma_s\|$ is bounded by some constant independent of $T$. The latter assumption is not unreasonable, as $\|\Gamma_s\|$ would be bounded by the norm of an occupation kernel for the straight line trajectory spanning the diameter of our feature space. With these assumptions, we have $\|\Gamma_{j,\phi}\|_H,\|\Gamma_{j,\psi}\|_H\leq \|\Gamma_{s}\|_H+\varepsilon$. Moreover, by Cauchy Schwarz on $|\langle\phi, \Gamma_{\hat{r}_{j,\phi}}\rangle|$ we have $\left\|\begin{pmatrix}
D_{1,\phi} + \langle \phi, \Gamma_{\hat r_{1,\phi}} \rangle,\cdots, D_{s,\psi} + \langle \phi, \Gamma_{\hat r_{s,\phi}}\rangle\end{pmatrix}^{\top}\right\|_2$ is bounded by some constant $L$. This above discussion culminates in the following lemma. 
\begin{lemma}\label{term1}
With the above assumptions, 
\[\left\|\begin{pmatrix}\hat{w}_{1,\phi} \\ \vdots \\ \hat{w}_{s,\phi}\end{pmatrix}\right\|_2\leq 
\|G_\phi^{-1}\|\left\|\begin{pmatrix}
D_{1,\phi} + \langle \phi, \Gamma_{\hat r_{1,\phi}} \rangle\\ \vdots \\ D_{s,\phi} + \langle \phi, \Gamma_{\hat r_{s,\phi}}\rangle
\end{pmatrix}\right\|_2\leq L\cdot\frac{2\mu}{C_2}\frac{q_{X,T,\phi}^2}{\exp(-M_2^2/(q_{X,T,\phi}^2\mu))T^2}
\]
where $L$ is independent of $T$.
\end{lemma}

Our final assumption is that
 \[\|\Gamma_{r_\phi}-\Gamma_{r_\psi}\|\leq CT\|\phi-\psi\|
 \]
for some constant $C$ independent of $T$. This is an assumption on the regularity on the solutions to the differential equations stated in terms of the occupation kernels. While more work is needed, its possible that this condition could be stated in terms of Frechet differentiability. 
\subsection*{Proof of our Main Theorem:}
We are now in a position to prove Theorem \ref{contraction mapping}
\begin{proof}
Consider, 
\begin{align*}\label{maininequality-1}\|A(\phi)-A(\psi)\|_H
=&\left\|\sum_{j=1}^s\hat{w}_{j,\phi}\Gamma_{\hat{r}_{j,\phi}}(x)-\sum_{j=1}^s\hat{w}_{j,\psi}\Gamma_{\hat{r}_{j,\psi}}(x)\right\|_H\nonumber\\
\leq& \sum_{j=1}^s|\hat{w}_{j,\phi}|\|\Gamma_{\hat{r}_{j,\phi}}-\Gamma_{\hat{r}_{j,\psi}}\|_H+\sum_{j=1}^s|\hat{w}_{j,\phi}-\hat{w}_{j,\psi}|\|\Gamma_{\hat{r}_{j,\psi}}\|_H\
\end{align*}
By our assumptions and Lemma \ref{term1} we have that the first term is proportional to $\frac{1}{T}\|\phi-\psi\|_H$. By inequalities \eqref{norm1}, \eqref{norm2} and a reevaluation of our estimates \eqref{estimate1} and \eqref{estimate2} given the assumptions, the second term in the above is also proportional to $\frac{1}{T}\|\phi-\psi\|_H$. Thus for a suitable $T$ we have a contraction mapping and by Contraction Mapping Theorem our algorithm will converge.
\end{proof}
 \section{Numerical Experiments}
For our first experiment, Figure \ref{Experiment1}, we generated a slope field $F(x)$ using a linear combination of Gaussian kernels and generated a set of random points and angles to serve as our anticipated dynamics using unit speed. 
\[F(x)=\frac{1}{8}\begin{pmatrix}f_1(x)\\f_2(x)\end{pmatrix}\]
\begin{align*}f_1(x)&=
            5\exp\left(-2\left\|x-\begin{bmatrix}.25\\.25\end{bmatrix}\right\|^2\right)- .2\exp\left(-\left\|x-\begin{bmatrix}.25\\.75\end{bmatrix}\right\|^2\right)\\ 
            &+ 2\exp\left(-\left\|x-\begin{bmatrix}.75\\.75\end{bmatrix}\right\|^2\right)- 5\exp\left(-2\left\|x-\begin{bmatrix}.75\\.25\end{bmatrix}\right\|^2\right)\\
            f_2(x)&=
            3\exp\left(-\left\|x-\begin{bmatrix}.25\\.25\end{bmatrix}\right\|^2\right)+\exp\left(-\left\|x-\begin{bmatrix}.25\\.75\end{bmatrix}\right\|^2\right)\\ 
            &- 3\exp\left(-3\left\|x-\begin{bmatrix}.75\\.75\end{bmatrix}\right\|^2\right)+ \exp\left(-\left\|x-\begin{bmatrix}.75\\.25\end{bmatrix}\right\|^2\right)
\end{align*}

\begin{figure}[ht]
\centering
\begin{subfigure}{.5\textwidth}
  \centering
  \includegraphics[width=\linewidth]{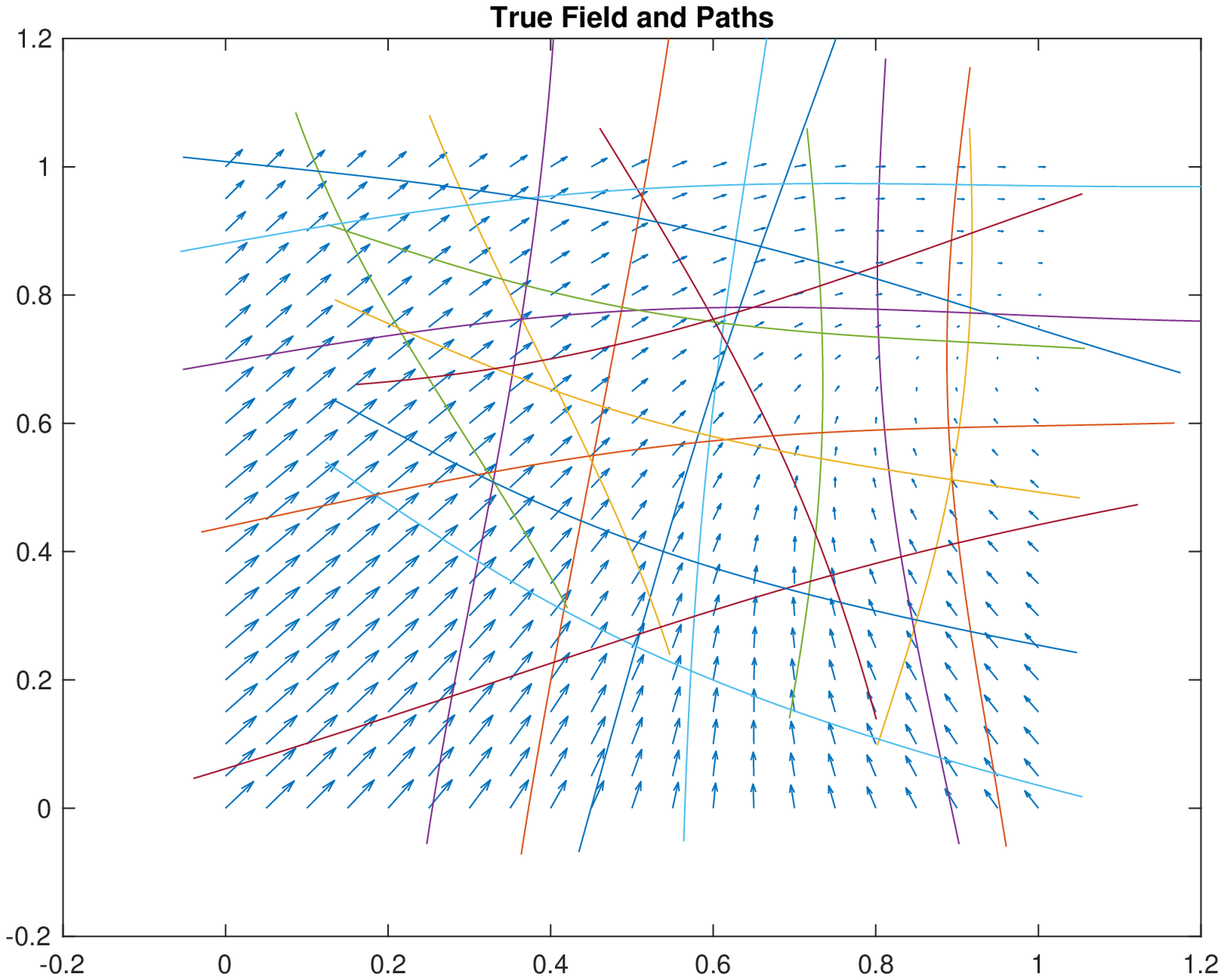}
  \caption{True vector field and trajectories.}
  \label{generated field}
\end{subfigure}%
\begin{subfigure}{.5\textwidth}
  \centering
  \includegraphics[width=\linewidth]{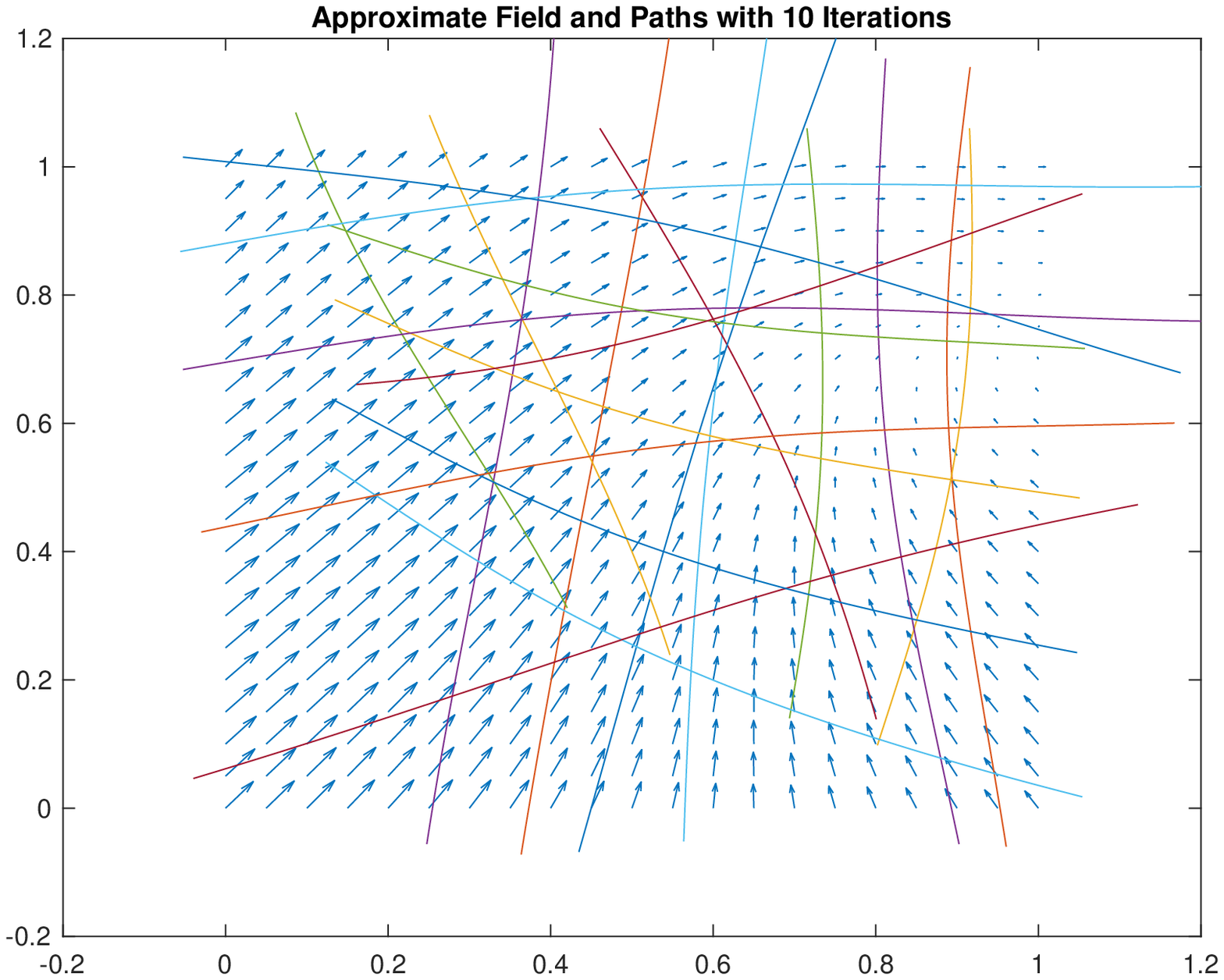}
  \caption{Results of Experiment 1.}
  \label{10iterates-gen}
\end{subfigure}
\caption{The true trajectories are calculated via RK4 over the time frame $[0,1]$. Using Gaussian RBFs with a kernel width of 1, we performed $10$ iterations of Algorithm \ref{alg}.}
\label{Experiment1}
\end{figure}
For our second experiment, Figures \ref{init},\ref{5iterates-real},\ref{10iterates-real}, and \ref{20iterates-real}, we used the Gliderpalooza 2013 data first presented in \cite{SCC.Chang.Wu.ea2017}. The data was for 31 sequential trajectory segments and contained the initial positions, the true final positions, dead-reckoned final positions, speeds and dead-reckoning times. For their experiments, they used averaged speeds and an average dead-reckoning time of 3.5 hours.

Algorithm \ref{alg} is designed to work on multiple trajectories necessitating some prior calculations. The data was scaled by a factor of $10^{-4}$, then the provided initial positions and dead-reckoned final positions were used to calculate the initial directions along with calculated speeds using their averaged dead-reckoning time of $3.5$ hours. The sequential trajectory data was broken in to $31$ separate trajectories. Using exponential kernels, we performed 5, 10, and 20 iterations of Algorithm \ref{alg} using a kernel width of $\mu=1/170$ for the $5$ iteration run and a kernel width of $\mu=1/10,000$ for the $10$ and $20$ iteration runs. The kernel widths for each of these runs were selected to produce well conditioned Gram matrices. 
\begin{figure}[ht]
\centering
\begin{subfigure}{.5\textwidth}
  \centering
  \includegraphics[width=\linewidth]{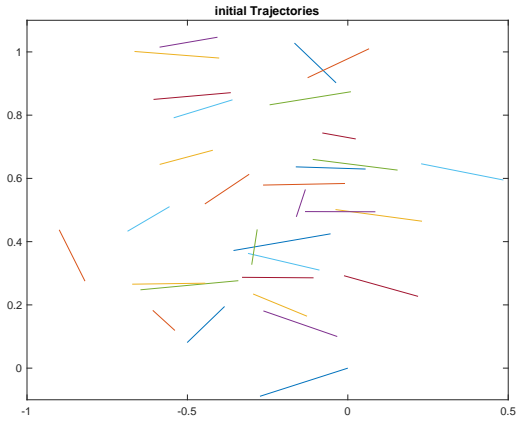}
  \caption{Initial trajectories.}
  \label{init}
\end{subfigure}%
\begin{subfigure}{.5\textwidth}
  \centering
  \includegraphics[width=\linewidth]{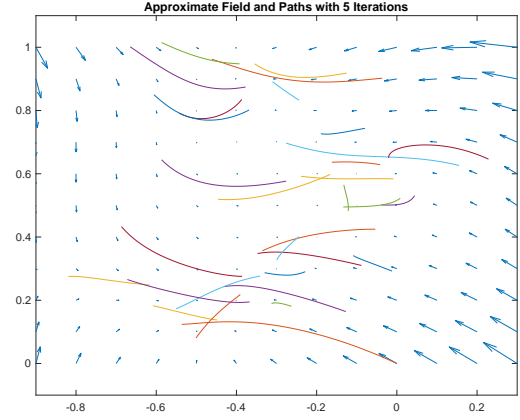}
\caption{Approximated field and trajectories after\\
\hphantom{\quad\ } 5 iterations.}
\label{5iterates-real}%
\end{subfigure}
\begin{subfigure}{.5\textwidth}
\includegraphics[width=\linewidth]{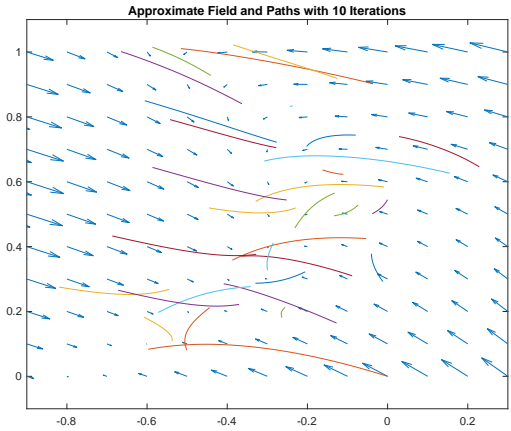}%
\caption{Approximate field and trajectories with\\ \hphantom{\quad\ }10 iterations}%
\label{10iterates-real}%
\end{subfigure}%
\begin{subfigure}{.5\textwidth}
\centering
\includegraphics[width=\linewidth]{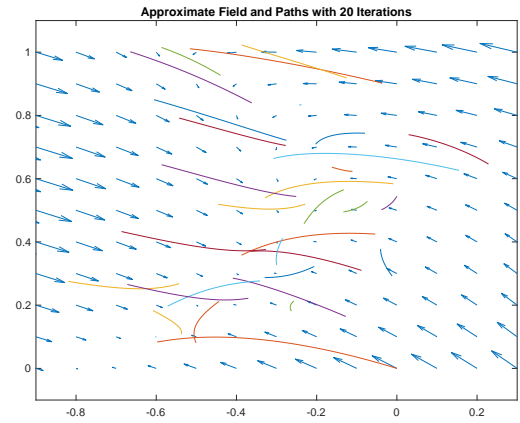}%
\caption{Approximate field and trajectories with\\\hphantom{\quad\ } 20 iterations}%
\label{20iterates-real}%
\end{subfigure}
\caption{Calculated initial trajectories and output of Algorithm 1 for 5, 10, and 20 iterations on Gliderpalooza data.}
\end{figure}

\section{Discussion}
This approach to motion tomography has several advantages over that of \cite{SCC.Chang.Wu.ea2017}. In the context of this aim, the flow field is approximated using the occupation kernels as basis functions for approximation, whereas \cite{SCC.Chang.Wu.ea2017} requires a piece-wise constant description of the flow field or a parameterization with respect to Gaussian RBFs. Moreover, \cite{SCC.Chang.Wu.ea2017} renormalizes the integral by multiplying and dividing by $\|\dot{\tilde{r}}(t)\|_2$ to artificially convert the integrals to line integrals. The renormalization may ultimately produce divide by zero errors that leads to stronger assumptions on the dynamical systems. The occupation kernel method avoids the renormalization and does not add further restrictions on the dynamics. Finally, the representation of the flow field with respect to the occupation kernel basis allows for the application of the approximation powers of RKHSs. It should also be noted that although the applications of this technique apply mainly to $\mathbb{R}^2$ valued functions, there is no inherit reason to limit to $\mathbb{R}^2$. That is, this technique would apply to $\mathbb{R}^d$ valued functions. 
For comparison, Figure \ref{dongsik-gendata-result} shows the output of the method in \cite{SCC.Chang.Wu.ea2017} from the data in Experiment 1. Figures \ref{error_dongsik}, \ref{error_kernel} are the plots of the difference between the true vector field and the approximated vectorfields for both methods.
\begin{figure}[H]
\centering
\begin{subfigure}{.5\linewidth}
    \centering
    \includegraphics[width=\linewidth]{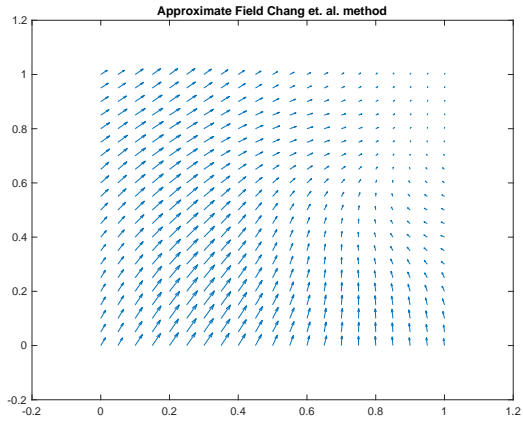}%
    \caption{Output of method from \cite{SCC.Chang.Wu.ea2017}}%
    \label{dongsik-gendata-result}%
\end{subfigure}%
\begin{subfigure}{.5\linewidth}
    \centering
    \includegraphics[width=\linewidth]{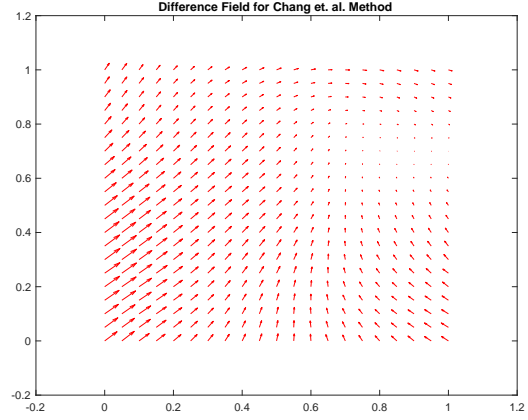}%
    \caption{Difference between the true and approximate field given by the method in \cite{SCC.Chang.Wu.ea2017}}%
    \label{error_dongsik}%
\end{subfigure}
\begin{subfigure}{.5\linewidth}
    \centering%
    \includegraphics[width=\linewidth]{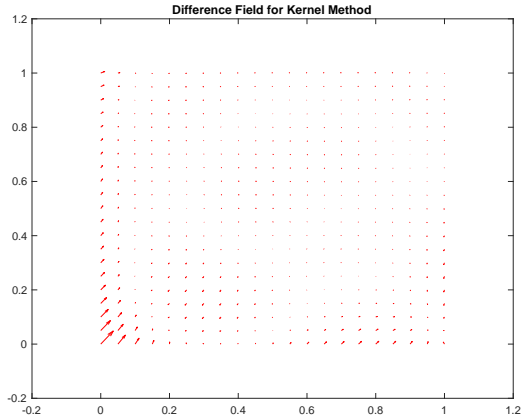}%
    \caption{Difference between the true and approximate field given by Algorithm \ref{alg}.}%
    \label{error_kernel}%
\end{subfigure}%
\begin{subfigure}{.5\linewidth}
    \centering%
    \includegraphics[width=\linewidth]{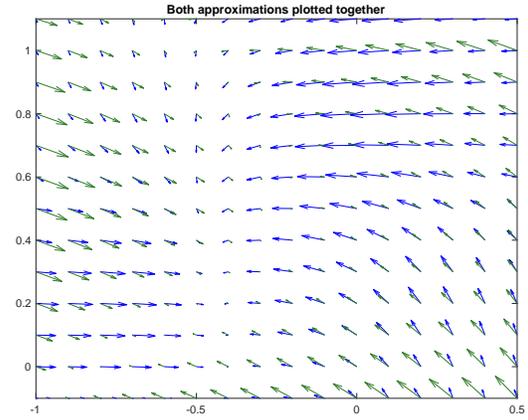}%
    \caption{Simultaneous plots, the output of Algorithm \ref{alg} is in green and the output of the method of \cite{SCC.Chang.Wu.ea2017} is shown in blue.}%
    \label{both}%
\end{subfigure}
\caption{Accuracy comparison between the method of Chang et. al. and Algorithm \ref{alg}.}
\end{figure}

Due to its structure for spatial representation of the flow-field, the method in \cite{SCC.Chang.Wu.ea2017} leads to poor performance for the region where sufficient trajectory information is not available. Quantitatively, given a sample of vectors $\{V(x_1,y_1),\ldots, V(x_n,y_n)\}$ and $\{W(x_1,y_1),\ldots, W(x_n,y_n)\}$ from two vector fields $V(x,y)$ and $W(x,y)$, we can define the max error (relative to $V$) as \[\text{Max Error}=\max_i \{\|V(x_i,y_i)-W(x_i,y_i)\|/\|V(x_i,y_i)\| \}\]
and the mean error (relative to $V$) as 
\[\text{Mean Error}=\operatorname{Mean} \{\|V(x_i,y_i)-W(x_i,y_i)\|/\|V(x_i,y_i)\|\}.\]
Let $V(x,y)$ represent the true vector field and let $W(x,y)$ denote the estimated vector field. Given the samples shown above, the max and mean errors for both methods are summarized in Table \ref{tab:errortable}.
\begin{table}[ht]
\centering%
\def\arraystretch{1.25}
\begin{tabular}{l|cc}
\ \ & Method of Chang et al. & Algorithm \ref{alg}\\
\hline
Max Error    & 1.1849 &  0.25321 \\
\hline
Mean Error   &0.51549 & 0.025642\\
\end{tabular}
\caption{Max and Mean Errors for the two methods.}
\label{tab:errortable}
\end{table}

Moreover, we have plotted number of iterates vs the Mean Error. The results can be seen in Figure \ref{accuracy}.
\begin{figure}[ht]
    \centering
    \includegraphics[width=.75\linewidth]{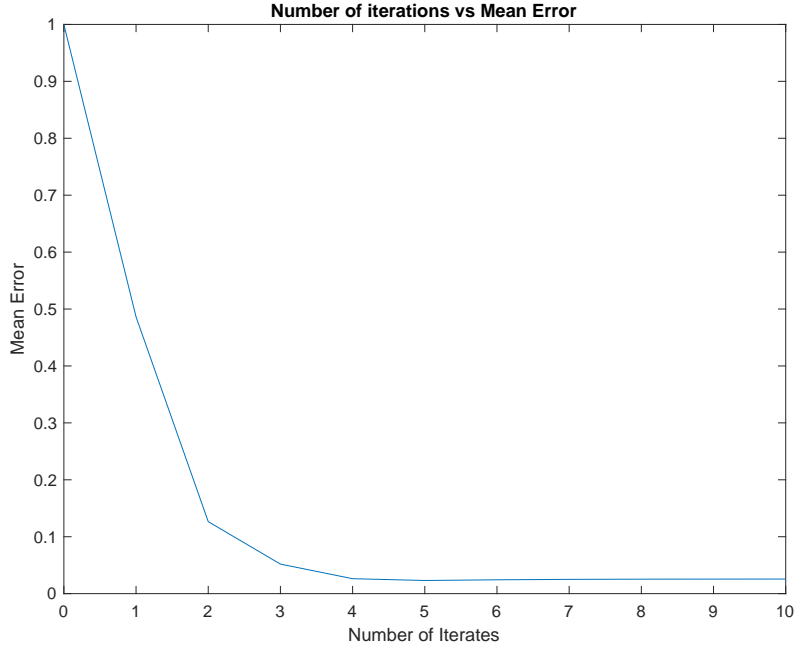}
    \caption{Number of iterates vs the Mean Error.}
    \label{accuracy}
\end{figure}
In the absence of ground truth in Experiment 2 we show a simultaneous plot of approximated fields from both Algorithm \ref{alg} and the method in \cite{SCC.Chang.Wu.ea2017} in Figure \ref{both}. For the set of samples depicted in Figure \ref{both}, we calculated the maximum norm difference, the mean norm difference, and the variance of the norm differences between the two fields. This is summarized in Table \ref{tab:stats}.
\begin{table}[ht]
\centering%
\def\arraystretch{1.25}
\begin{tabular}{l|l}
max norm difference   & 0.14877 \\
\hline
mean norm difference  &0.0088628  \\
\hline
variance& 0.0001869\\
\end{tabular}
\caption{Norm difference statistics between the two approximated fields.}
\label{tab:stats}
\end{table}

\bibliographystyle{plain}
\bibliography{bib}

\begin{thebibliography}{10}

\bibitem{SCC.Aweiss.Owens.ea2018}
Arwa~S. Aweiss, Brandon~D. Owens, Joseph Rios, Jeffrey~R. Homola, and
  Christoph~P. Mohlenbrink.
\newblock Unmanned aircraft systems ({UAS}) traffic management ({UTM}) national
  campaign {II}.
\newblock In {\em AIAA Inf. Syst.-AIAA Infotech@ Aerosp.}, page 1727, 2018.

\bibitem{geo-nonlinear-analysis}
Yoav Benyamini and Joram Lindenstrauss.
\newblock {\em Geometric nonlinear functional analysis. {V}ol. 1}, volume~48 of
  {\em American Mathematical Society Colloquium Publications}.
\newblock American Mathematical Society, Providence, RI, 2000.

\bibitem{burden2001numerical}
R.L. Burden, R.L. Burden, and J.D. Faires.
\newblock {\em Numerical Analysis}.
\newblock Number v. 1 in Numerical Analysis. Brooks/Cole, 2001.

\bibitem{SCC.Chang.Wu.ea2017}
Dongsik Chang, Wencen Wu, Catherine~R Edwards, and Fumin Zhang.
\newblock Motion tomography: Mapping flow fields using autonomous underwater
  vehicles.
\newblock {\em The International Journal of Robotics Research}, 36(3):320--336,
  2017.

\bibitem{comp-functional-analysis}
Ramon~E. Moore.
\newblock {\em Computational functional analysis}.
\newblock Ellis Horwood Series: Mathematics and its Applications. Ellis Horwood
  Ltd., Chichester; Halsted Press [John Wiley \& Sons, Inc.], New York, 1985.

\bibitem{paulsen2016introduction}
Vern~I. Paulsen and Mrinal Raghupathi.
\newblock {\em An introduction to the theory of reproducing kernel {H}ilbert
  spaces}, volume 152 of {\em Cambridge Studies in Advanced Mathematics}.
\newblock Cambridge University Press, Cambridge, 2016.

\bibitem{SCC.Petrich.Woolsey.ea2009}
Jan Petrich, Craig~A. Woolsey, and Daniel~J. Stilwell.
\newblock Planar flow model identification for improved navigation of small
  {AUV}s.
\newblock {\em Ocean Eng.}, 36(1):119--131, 2009.

\bibitem{rosenfeld2019dynamic}
Joel~A Rosenfeld, Rushikesh Kamalapurkar, L~Gruss, and Taylor~T Johnson.
\newblock Dynamic mode decomposition for continuous time systems with the
  liouville operator.
\newblock {\em arXiv preprint arXiv:1910.03977}, 2019.

\bibitem{rosenfeld2019occupation}
Joel~A Rosenfeld, Benjamin Russo, Rushikesh Kamalapurkar, and Taylor~T Johnson.
\newblock The occupation kernel method for nonlinear system identification.
\newblock {\em arXiv preprint arXiv:1909.11792}, 2019.

\bibitem{SCC.Stepanyan.Krishnakumar2017}
Vahram Stepanyan and Kalmanje~S. Krishnakumar.
\newblock Estimation, navigation and control of multi-rotor drones in an urban
  wind field.
\newblock In {\em AIAA Inf. Syst.-AIAA Infotech @ Aerosp.}, page 0670, 2017.

\bibitem{Wolfgang}
R.~Thompson and W.~Walter.
\newblock {\em Ordinary Differential Equations}.
\newblock Graduate Texts in Mathematics. Springer New York, 2013.

\bibitem{Wendland}
Holger Wendland.
\newblock {\em Scattered data approximation}, volume~17.
\newblock Cambridge university press, 2004.

\bibitem{wu2013glider}
Wencen Wu, Dongsik Chang, and Fumin Zhang.
\newblock Glider ct: Reconstructing flow fields from predicted motion of
  underwater gliders.
\newblock In {\em Proceedings of the eighth ACM international conference on
  underwater networks and systems}, pages 1--8, 2013.

\end{thebibliography}

\end{document}